\documentclass[12pt]{article}
  \usepackage[centertags]{amsmath}
  \usepackage{amsfonts}
  \usepackage{amssymb}
  \usepackage{amsthm}
  \usepackage{newlfont}

  \newlength{\defbaselineskip}
  \newcommand{\setlinespacing}[1]%
                               {\setlenght{\baselineskip}{#1 \defbaselineskip}}

  \theoremstyle{plain}
  \newtheorem{thm}{Theorem}[section]
  \newtheorem{cor}[thm]{Corollary}

  \newtheorem{rem}[thm]{Remark}
  
  \theoremstyle{definition}
  \newtheorem{defi}[thm]{Definition}
  \newtheorem{exm}[thm]{Example}
  
  \numberwithin{equation}{section}

\begin{document}
\begin{center}
{\bf Minimal projections with respect to various norms}
\end{center}
\vspace{.15 cm}
\begin{center}
\small{Asuman G\"{u}ven AKSOY and Grzegorz LEWICKI}
\end{center}

  \date{October 7, 2009}
\mbox{~~~}\\
\mbox{~~~}\\
\small\mbox{~~~~}{\bf Abstract} {\footnotesize  We will show that a theorem of Rudin \cite{wr1}, \cite{wr},
permits us to determine minimal projections not only with respect to the operator norm but with respect to quasi-norms in operators ideals and numerical radius in many concrete cases.  }\\

\
\\

\section{ Introduction.}
Let $X$ be a Banach space over $\mathbb{R}$ or $\mathbb{C}$. We write $B_{X}$ for
the closed unit ball and $S_{X}$ for the unit sphere of $X$.
 The dual space is denoted by $X^{*}$ and the Banach algebra of all
continuous linear operators on $X$ is denoted by $B(X)$.
\begin{defi}
The \textit{numerical range} of $T\in B(X)$ is defined by
$$W(T)= \{ x^{*}(Tx)  :~x\in S_{X},~x^{*}\in S_{X^{*}},~x^{*}(x)=1\}\cdot$$
The \textit{numerical radius} of $T$ is then given by
$$\parallel T \parallel_{w}=\sup\{\vert \lambda\vert : ~\lambda\in W(T)\}\cdot$$
\end{defi}
Clearly, $\parallel . \parallel_{w}$ is a semi-norm on $B(X)$ and $\parallel T \parallel_{w} \le \Vert T\Vert$ for
all $T\in B(X)$.
\\ The \textit{numerical index} of $X$ is defined by
$$n(X)=\inf\{\parallel T \parallel_{w} :~ T\in S_{B(X)}\}\cdot$$
Equivalently, the numerical index $n(X)$ is the greatest constant $k
\geq 0$ such that $k\|T\| \leq \parallel T \parallel_{w}$ for every $T \in B(X)$. Note
also that $0 \leq n(X) \leq 1$, and $n(X) > 0$ if and only if
$\parallel \cdot \parallel_{w}$
and $\|\cdot\|$ are equivalent norms.\\
The concept of numerical index was first introduced by Lumer
\cite{lg} in 1968. Since then much attention has been paid to this
constant of equivalence between the numerical radius and the usual
norm in the Banach algebra of all bounded linear operators of a
Banach space. Classical references here are  \cite{bff-dj1},
 \cite{bff-dj2}. For recent results we refer the reader to
\cite{aag-cbl},\cite{aga-ed-kham} \cite{ee},
\cite{fc-mm-pr},\cite{gke-rdkm}, \cite{lg-mm-pr}, \cite{mm}.\\
   Let $\mathcal{L}$ denote the class of all operators between arbitrary Banach spaces. An \emph{Operator ideal} $\mathcal{U}$ is, roughly speaking, a subclass of
$\mathcal{L}$ such that $$\mathcal{U}+ \mathcal{U}= \mathcal{U}\,\,\, \mbox{and}\,\,\,\mathcal{L} \circ \mathcal{U}\circ \mathcal{L}= \mathcal{U}$$ The theory of normed operator ideals was founded by A. Grothendick and R. Schatten. Basic examples are the ideals of nuclear, compact, and absolutely summing operators. For more details on operator ideals see \cite{ap}. On every operator ideal there could be many different quasi-norms; however the ``good" quasi-norms are selected by the completeness of the corresponding topology. Therefore, one can state that  every operator ideal up to equivalence has one reasonable quasi-norm.
\begin{defi}
\label{opiddef}
Let $\mathcal {U}$ be an operator ideal. A map $$\mathcal{Q} : \mathcal {U}\rightarrow \mathbb{R}^+ $$ is called a \emph{quasi-norm} if the following conditions are satisfied:
\begin{enumerate}
\item $\mathcal{Q}( I_V)= 1$, where $V$ denotes one dimensional Banach space.
\item There exists a constant $C \geq 1$ such that
\begin{equation}
 \label{C}
\mathcal{Q}(S_1+S_2) \leq C [\mathcal{Q}(S_1)+\mathcal{Q}(S_2)]\,\,\, \mbox{for}\,\,\, S_1,S_2 \in \mathcal {U}(X,Y)
\end{equation}
\item If $T \in B(X_0, X), \,\, S\in B(X,Y)\,\, \mbox{and} \,\,R\in B(Y,Y_0)$, then $$ \mathcal{Q}(RST) \leq \| R \| \mathcal{Q}(S) \| T \|$$
\end{enumerate}
\end{defi}
Note that a quasi-operator ideal $(\mathcal U ,\mathcal{Q})$ is an operator ideal $\mathcal{U}$ with quasi-norm $\mathcal Q$ such that $\mathcal{Q}(\lambda S)= |\lambda| \mathcal{Q}(S)$ for $S \in \mathcal{U}(X,Y)$  with $\lambda \in \mathbb{R}$. Furthermore,  we have  $$ \| S \|_w \leq \| S \| \leq \mathcal{Q}(S)$$ for all $S \in \mathcal{U}$.
\\
If $X$ is a Banach space and $V$ is a linear, closed subspace of $X,$ by $\mathcal{P}(X,V)$ we denote the set of all linear projections continuous with respect to the operator norm. Recall that an operator $P:X\rightarrow V$ is called a projection, if $ P|_{V} = id_{V}.$
A projection $ P_o \in \mathcal{P}(X,V)$ is called minimal if
$$
\| P_o \| = \inf \{ \|P\|: P\in \mathcal{P}(X,V)\} = \lambda(V,X).
$$
Minimal projections were extensively studied by many authors in the context of functional analysis and approximation theory (see e.g.
\cite{aag-cbl}, \cite{BE1}, \cite{chalL1}-\cite{CM1}, \cite{CF}-\cite{CMO},\cite{FMW},\cite{IS1}-\cite{LZ1}, \cite{MI1}-\cite{OL},\cite{wr1},\cite{SK1}-\cite{SK3}). Mainly the problems of existence of minimal projections, uniqueness of minimal projections, finding concrete formulas for minimal projections and estimates of the constant $ \lambda(V,X)$ were considered.
It is worth noting that one of the main tools for finding minimal projection effectively is the so-called Rudin's Theorem (see \cite{wr1} or \cite{wr}).

Now assume that $X$ is a Banach space, $N$ is any norm or semi-norm on an operator ideal $\mathcal{U}(X) \subset B(X)$ and $V$ is a subspace of $X$ such that $id|_{V} \in \mathcal{U}(X).$
Denote by $ \mathcal{P}_N(X,V)$ the set of all linear projections from $X$ onto $V$ such that $P \in \mathcal{U}(X).$
Let us define
$$
\lambda_N(V,X)= \inf \{ N(P) : P \in \mathcal{P}_N(X,V) \}
.$$
We put $\lambda_N(V,X) = + \infty $ if $\mathcal{P}_N(X,V) = \emptyset$.
A projection $ P_o \in \mathcal{P}_N(X,V)$ is called N-minimal if
$$
N(P_o)= \lambda_N(V,X).
$$
In the following we consider $N$ as the numerical radius $\| . \|_w$  or the quasi-norm norm $\mathcal{Q}( . )$ and we will show that Rudin's Theorem can be applied to obtain N-minimal projections effectively. Although our proofs follow from Rudin's result without much difficulty,  applications given in the last section justifiys the reason one may want to study minimal projections in this context.
\\ It is worth mentioning that we do not know any paper (with the exception of \cite{aag-cbl}) concerning minimal projections with respect to norms different than the operator norm.
In fact, in \cite{aag-cbl} a characterization of minimal numerical-radius extensions of operators from a normed linear space $X$ onto its finite dimensional subspaces and comparison with minimal operator-norm extension are given.
\section{Main Results }

One of the key theorems for minimal projections is due to W. Rudin (see page 127 of \cite{wr1} or \cite{wr}). This theorem was motivated by an earlier result of
Lozinki\v{i}  (see e.g \cite{che}, p. 216 or \cite{LZ1}) concerning the minimality of the the classical n-th Fourier projection in $ \mathcal{P}(C(2\pi), \pi_n),$ where $C(2\pi)$
denotes the space of all $2\pi$-periodic real-valued functions equipped with the supremum norm and $\pi_n$ is the space of all trigonometric polynomials of degree less than or equal to $n.$
The setting for this theorem is as follows. $X$ is a Banach space and $G$ is a compact topological group. Defined on $X$ is a set $\mathcal {A}$ of all bounded linear, bijective operators in a way that $\mathcal {A}$  is algebraically isomorphic to $G$. The image of $g \in G$ under this isomorphism will be denoted by $T_g$. We will assume that the map $G\times X \rightarrow X$ defined as $(g,x)\mapsto T_{g}x$ is continuous. A subspace $V$ of $X$ is called invariant under $G$ if $T_g V\subset V$ for all $g\in G$.
\begin{thm}
\label{Rudin}
Let $X$ and $G$ satisfy the above hypotheses, and let $V$ be a closed subspace of $X$ which is invariant under $G$. If there exists a bounded projection $P$ of $X$ onto $V$, then there exists a bounded projection $Q$ of $X$ onto $V$ which commutes with $G$.
\end{thm}

 The idea of the proof of the above theorem is to obtain $Q$ by averaging the operators $T_g^{-1} P T_g$ with respect to Haar measure $\mu$ on $G$. For details see
\cite{wr}.
\\ Now assume $X$ has a norm which is an isometry for the maps in $\mathcal{A}$, and all the hypotheses of Theorem 2.1 are met, then we can claim the following corollary which is clearly a stronger result than Theorem 2.1.
\begin{cor}
\label{unicity}
If there is a \textbf{unique} projection $Q: X\rightarrow V$ which commutes with $G$, ($Q\circ T_g = T_g \circ Q$) and $V$ is separable
then for any $P \in \mathcal{P}(X,V)$  the projection $Q$ defined by

$$ Q x= \displaystyle  \int _{G} (T_{g} ^{-1}  P  T_g)x \, d\mu (g) $$
is a \textbf{minimal projection} of $X$ onto $V$ with respect to the operator norm.
Here $\mu$ denotes the normalized ($\mu(G) = 1)$ Haar measure on $G$ and the above integral is understood as the Bochner integral. (By our assumptions concerning Rudin's Theorem, for any $x \in X$ the function
$g \rightarrow (T_{g} ^{-1}  P  T_g)x$ is continuous on $G$ and by the compactness of $G$ is Bochner integrable.)
\end{cor}
For the proof of above corollary we refer to \cite{wr} and \cite{che-light}. For more details about Bochner integrals, see \cite{che-light}.
 Now we show that Theorem \ref{Rudin} can be applied to find $N$-minimal projections.
\begin{thm}
\label{numradius}
Suppose all the hypotheses of Theorem \ref{Rudin} are satisfied, $V$ is separable and the maps $\mathcal{A}$ are to be isometries. If $P$ is any projection in numerical radius from $X$ onto $V$, then the projection $Q_P$ defined as
$$
Q_P  x= \displaystyle  \int _{G} (T_{g}^{-1} P  T_g)x \,d\mu (g)
$$
satisfies
$$
\|Q_P\|_w \leq \|P\|_w.
$$
Here $\mu$ is the normalized Haar measure on $G$ and the above integral is understood as the Bochner integral.
\end{thm}
\begin{proof}
Consider $\| Q_P \|_w= \sup \{ | x^*(Q_Px) |: \,\, x^*(Q_Px) \in W(Q)\} $, and since
$$
| x^*(Q_Px) |= | x^*\,\, \displaystyle( \int _{G} (T_{g}^{-1}  P  T_g)x\,\, d\mu(g)|
$$
we have:
$$
| x^*(Q_Px) |= |\displaystyle \int _{G} (x^* \circ (T_g^{-1} ) P  (T_gx) \,d\mu (g)|
$$
$$
\leq \displaystyle \int _{G} |(x^*T_g^{-1} ) P  (T_gx)| \,d\mu (g).
$$
But $\| x \| =1$ and $\| x^* \| =1$ implies that $\| T_g x \| =1$ and $\| x^* \circ T_{g^{-1}} \| =1$; moreover $ (x^* \circ T_{g^{-1}}) (T_gx) = x^*(x)= 1$, and thus
$$
| x^*(Qx) |\leq \displaystyle \int _{G} \|P\|_w \,d\mu(g)\leq \| P\|_w
$$
and $ \| Q \|_w \leq \| P \|_w$ as required.
\end{proof}
\begin{thm}
\label{numminimal}
Suppose that all the hypotheses of Theorem \ref{numradius} are satisfied and that there exists exactly one projection $Q$ which commutes with $G.$ Then $Q$ is a minimal projection with respect to the numerical radius.
\end{thm}
\begin{proof}
 Let $ P \in \mathcal{P}(X,V).$ By the properties of the Haar measure, $Q_p$ given in Theorem \ref{numradius} commutes with $G.$ Since there is exactly one projection
which commutes with $G,$ $Q_p =Q. $ By Theorem \ref{numradius}, $\|Q\|_w \leq \|P\|_w, $ as required.
\end{proof}

Now we consider the case of quasi-norms in operator ideals.

\begin{thm}
\label{opideal}
Assume that $\mathcal{U}(X) \subset B(X)$ is an operator ideal and let $N$ be a quasi-norm on $\mathcal{U}(X).$
Suppose all the hypotheses of Theorem \ref{Rudin} are satisfied, $V$ is separable, $T_g \in \mathcal{U}(X)$ and $T_g$ is an isometry for any $g \in G.$  For
$P\in \mathcal{P}_N(X,V)$ define $ Q_P \in \mathcal{P}_N(X,V)$ as $$ Q_P= \displaystyle  \int _{G}  T_{g}^{-1}   P  T_g \,d\mu (g), $$
where $ \int _{G}  T_{g}^{-1}   P  T_g \,d\mu (g)$ is understood as a linear operator from $X$ into $X$ defined by
$$
( \int _{G}  T_{g}^{-1}   P  T_g \,d\mu (g))x =\int _{G}  (T_{g}^{-1}   P  T_g)x \,d\mu (g).
$$
Assume additionally that
\begin{equation}
\label{important}
N( \int _{G}  T_{g}^{-1}   P  T_g \,d\mu (g)) \leq \int _{G}N(T_{g}^{-1}   P  T_g) \,d\mu (g).
\end{equation}
Then $N(Q_P)\leq N(P).$  Here $\mu$ as before is the normalized Haar measure on $G$ and the above integral is understood as the Bochner integral.
\end{thm}
\begin{proof}
By the property (3) of Definition \ref{opiddef} for any $ g \in G$ we have:
$$
N( T_{g}^{-1}P_{g}T_g) \leq \|T_{g}^{-1}\| N(P)  \| T_g \|
$$
By our assumption
$$
N( \int _{G}  T_{g}^{-1}   P  T_g \,d\mu (g)) \leq \int _{G}N(T_{g}^{-1}   P  T_g) \,d\mu (g)) \leq N(P),
$$
as claimed.
\end{proof}
Now we show that the assumption (\ref{important}) is not very restrictive.
\begin{rem}
\label{opi}
Let $\mathcal{U}(X) \subset \mathcal{B}(X)$ be an operator ideal. Assume that for any $L_n, L \in \mathcal{U}(X)$ if
$L_n \rightarrow L$ pointwise $(\|L_nx - Lx\| \rightarrow 0 $ for any $x\in X),$  then
\begin{equation}
\label{crucial}
N(L) \leq limsup_n N(L_n).
\end{equation}
In this case the assumption (\ref{important}) is satisfied. It is not very difficult to check that (\ref{crucial}) is satisfied for operator norms, p-summing norms,
p-nuclear norms and p-integral norms (for precise definitions see e.g. (\cite{WO1}, Chapter III F).
\end{rem}
\begin{thm}
\label{opminimal}
Suppose that all the hypotheses of Theorem \ref{opideal} are satisfied and that there exists only one projection $Q$ which commutes with $G.$ Then $Q$ is an $N$- minimal projection.
\end{thm}
\begin{proof}
The proof follows the same lines as the proof of Theorem \ref{numminimal}.
\end{proof}
Now we show that the minimal projection obtained above by application of Theorem \ref{Rudin} is also a cominimal projection.
\begin{thm}(Cominimal projections)
\label{cominimal-1}
Assume that the assumptions of Theorem \ref{numminimal} are satisfied and let $I_{X}$ denoted the identity operator on $X$. Then
$$
N( I_{X}- Q) \leq N( I_{X}-P)
$$
for any $P \in \mathcal{P}_N(X,V),$ where $Q$ is defined in Theorem \ref{numminimal}.
\end{thm}
\begin{proof} Let $Q_P$ be as in Theorem \ref{numradius}. Then $Q_P = Q.$ Hence
$$
N( I_{X}- Q ) = N(I_{X}-  \displaystyle \int _{G} T_{g}^{-1} P  T_g d\mu(g))
$$
$$
 =  N(I_{X} \displaystyle \int _{G} d\mu (g) - \displaystyle \int _{G} T_{g}^{-1} P T_gd\mu(g))=  N(\displaystyle \int _{G} T_{g}^{-1}(I_X- P) T_gd\mu(g)).
$$
Now fix any $(x^*,x)\in S_{X^*}\times S_X.$
Then
$$ | x^*(I_{X}- Q)x |= |  \displaystyle \int _{G} x^*\circ (T_{g}^{-1} (Id_X-P )T_g)x) d\mu(g)|
$$
$$ \leq \int _{G} |x^*\circ(T_{g}^{-1} (Id_X-P )T_g)x)| d\mu(g)
\leq \int _{G} |(x^*\circ T_{g}^{-1}) (I_X-P)   (T_gx)|  d\mu (g).
$$
Note that $\| x \| =1$ and $\| x^* \| =1$ implies that $\| T_g x \| =1$ and $\| x^* \circ T_{g^{-1}} \| =1$. Moreover $ (x^* \circ T_{g^{-1}} (T_gx) = x^*(x)= 1$. Thus
$$N( I_{X}- Q) \leq \displaystyle \int _{G}N(Id_X-P) d\mu (g)
                 \leq N(Id_X-P),
$$
which completes the proof.
\end{proof}

Combining the proofs of Theorem \ref{opideal} and Theorem \ref{cominimal-1} one can easily get the following result.

\begin{thm}(Cominimal projections)
\label{cominimal-2}
Assume that the assumptions of Theorem \ref{opminimal} are satisfied and let $I_{X}$ denoted the identity operator on $X$. If $Id_X \in \mathcal{U}(X),$ then
$$
N( I_{X}- Q) \leq N( I_{X}-P)
$$
for any $P \in \mathcal{P}_N(X,V),$ where $Q$ is defined in Theorem \ref{opminimal}.
\end{thm}
\begin{rem}
\label{rem1}
In Theorem \ref{cominimal-1} and Theorem \ref{cominimal-2} we can replace $Id_X$ by any operator $S \in \mathcal{U}(X)$ satisfying
$$
S \circ T_g = T_g \circ S
$$
for any $ g \in G.$ Hence
$$
N(S-Q) \leq N(S-P)
$$
for any $P \in \mathcal{P}_N(X,V).$
\end{rem}
\begin{rem}
\label{W}
 Assume that we have $W \subset S(X^*)\times S(X)$ with the following properties:
\begin{itemize}
\item For any $T:X\rightarrow X$, $T \in \mathcal{A}$ ( $T$ is an isometry)
\item If $(x,x^*) \in W$ then $(x^*T^{-1}, Tx ) \in W $
\end{itemize}
Define on $ \mathcal{L}(X)$ a semi-norm $ \| \cdot \|_W$ given by
$$
\|L\|_W= \sup\{ |x^*Lx|: (x^*,x) \in W\}
$$
for any $L \in \mathcal{L}(X).$
Observe that for
$$
W = \{(x^*,x)\in S_{X^*}\times S_X:x^*(x)=1 \},
$$
$ \| L \|_W$ is equal to the numerical radius of $L.$
Then the semi-norm $\| . \|_W$ satisfies Theorem \ref{numradius}, Theorem \ref{numminimal}, Theorem \ref{cominimal-1}
and Remark \ref{rem1}.
\end{rem}

\section{ Applications}
We start with a classical example which explains the origins of the Rudin Theorem.
\begin{exm}
\label{ex1}
Let $C(2\pi)$ denote the set of all continuous, $2\pi$-periodic functions and $\mathcal{\pi}_n$ be the space of all trigonometric polynomials of order $\leq n$ ( $n\geq
1)$. A \textbf{Fourier projection} $F_n :C(2\pi) \rightarrow \mathcal{\pi}_n$ is defined by the formula:
 $$ F_n(f) = \displaystyle \sum _{k=0}^{2n} ( \displaystyle \int_{0}^{2 \pi} f(t)\, g_n(t)dt )\,\, g_k$$ where $(g_k)_{k=0}^{2n}$ is an orthonormal basis in
$\mathcal{\pi}_n$ with respect to the scalar product
$$
<f,g>=\int_{[0,2\pi]}f(t)g(t)dt.
$$
In \cite{LZ1} it is shown that $F_n$ is a minimal projection in $\mathcal{P}(C(2\pi),\pi_n).$ The method of the proof is based on the Marcinkiewicz equality (see \cite{che}, p.233). For any $P \in \mathcal{P}(C(2\pi),\pi_n),$ $ f \in C(2\pi)$ and $t \in [0,2\pi]$, let
$$
F_nf(t) = (1/2\pi)\int_{[0,2\pi]}(T_{g^{-1}}P T_g f)t d \mu(g).
$$

Here $\mu$ is the Lebesgue measure and $(T_gf)t = f(t+g)$ for any $ g \in \mathbb{R}.$
Notice that $F_n$ is the only projection which commutes with $G,$ where $G = [0,2\pi)$ with addition mod $2\pi.$
Consequently $F_n$ is an $N$-minimal projection as considered in Theorem \ref{numminimal} and Theorem \ref{opminimal}.

Furthermore, it is known that (see\cite{che},page 212) the operator norm of $F_n$ satisfies the following:
 $$ \displaystyle \frac{4}{\pi^2} \ln(n) \leq \| F_n \| \leq \ln(n)+3. $$ In \cite{aag-cbl}, it is shown that in cases of $L^p$ , $p= 1,\infty$, numerical
radius extensions and minimal norm extensions are equal. Since $C(2\pi) \subset L^\infty$, we also have
$$
\frac{4}{\pi^2} \ln(n) \leq \| F_n \|_w \leq \ln(n)+3.
$$
It is worth noting that the Marcinkiewicz equality holds true if we replace $C(2\pi)$ by $L^p[0,2\pi]$ for $1\leq p \leq \infty$ or by Orlicz space $L^{\phi}[0,2\pi] $
equipped with the Luxemburg or the Orlicz norm provided that $\phi$ satisfies the suitable $\Delta_2$ condition. Hence, Theorem \ref{numminimal} and Theorem \ref{opminimal}
can be applied to the numerical radius and quasi-norms generated by $L_p$ norm or the Luxemburg or the Orlicz norm.
\end{exm}

Now we consider a more general situation.

\begin{exm}
\label{ex2}
Let $m,n \in \mathbb{N},$ $ n < m.$ Assume
$$
V = span[sin(k_i\cdot),cos(k_i\cdot), i=1,...,n]
$$
and let
$$
X= span[sin(k_i\cdot),cos(k_i\cdot), i=1,...,m],
$$
where $ k_i \in \mathbb{N}$ and $k_1 < k_2... <k_m.$ Assume that $G$ is as in Example \ref{ex1}. It is easy to see that the only projection from $X$ onto $V$
which commutes with $G$ is given by
$$
Q(sin(k_i\cdot)=0, Q(cos(k_i\cdot)=0
$$
for $ i >n.$
Assume that $\| \cdot\|_X $ is any norm on $X$ such that the mapping
$$
 T_g:(X, \| \cdot \|_X) \rightarrow  T_g:(X, \| \cdot \|_X)
$$
is a linear isometry for any $g \in G. $
Then $Q$ is a $ N$-minimal projection as considered in Theorem \ref{numminimal} and Theorem \ref{opminimal}.
Typical examples of $\|\cdot \|_X$ are the $L_p$-norms, the Luxemburg and the Orlicz norms.
Also it is possible to replace $X$ by $L_p[0,2\pi]$ for $1 \leq p \leq \infty$ or by Orlicz spaces $L^{\phi}[0,2\pi].$
\\
The same situation holds true in the complex case with
$$
X=Span[e^{ik_jt}:i=1,...,m],
$$
$$
V=Span[e^{ik_jt}:i=1,...,n]
$$
and
$$
G= \{ e^{it}:t \in [0,2\pi]\}.
$$
Also we can apply Theorem \ref{numminimal} and Theorem \ref{opminimal} in multi-dimensional settings (see e.g \cite{LE1}).
\end{exm}
\begin{exm}
\label{ex3}
Let $X = L^p[0,2\pi] $ and let $V = H^p[0,2\pi]$ be the Hardy space for $1 < p < \infty.$ By the M. Riesz Theorem (see \cite{wr}, p.152), it follows that $\mathcal{P}(L^p[0,2\pi], H^p[0,2\pi]) \neq \emptyset$ and that the projection $Q$ given by
$$
Q(e^{ik_j\cdot})=0
$$
for $j<0$ is the only projection which commutes with $G= \{ e^{it}:t \in [0,2\pi]\}.$ Hence $Q$ is an $N$-minimal projection as considered in Theorem \ref{numminimal} and Theorem \ref{opminimal}.
\end{exm}
\begin{exm}
\label{ex4}
Let $M(n,m)$ be the space of all (real or complex) matrices of $n$ rows and $m$ columns. Denote by $M(n,1)$ ($M(1,m)$ respectively) the space of matrices from $M(n,m)$ with constant rows (constant columns respectively). Let $S_n$ be the group of permutations of the set $\{1,...,n\}.$ Let $G = S_n \times S_m.$ For any $ g = \sigma \times \gamma \in G$ define a mapping $T_g: M(n,m) \rightarrow M(n,m)$ by
$$
T_g(A)(i,j) = A(\sigma(i),\gamma(j))
$$
for any $ A \in M(n,m),$ $i=1,...,n$ and $j=1,...,m.$
Let
$$
W= M(n,1)+M(1,m).
$$
It is easy to see that $T_g(W) \subset W$ for any $g \in G.$ Now assume that $$X= (M(n,m), \| \cdot \|)$$ where $\| \cdot\|$ is any norm such that the mappings $ T_g$ are isometries on $G.$ Typical examples of such norms are $l_p$-norms and the Luxemburg and Orlicz norms. In (\cite{che-light}, Chapter 9) it has been shown that there is the unique projection $Q$ which commutes with $G$ that is given by a formula
$$
Qe_{rs}(i,j)=
\begin{cases}\frac{n+m+1}{nm} &  i=r,j=s \\
\frac{m-1}{nm} & i\neq r, j=s\\
\frac{n-1}{nm} & i=r, j\neq s\\
\frac{-1}{nm} & i\neq r, j\neq s
             \end{cases}
$$
where $e_{rs}(i,j) =\delta_{ri}\delta_{rj}.$ Hence $Q$ is an $N$-minimal projection in any case considered in Theorem \ref{numminimal} and Theorem \ref{opminimal}.
\end{exm}
\begin{exm}
\label{ex5}
Let for $x \in \mathbb{R}$, $[x]$ denote the integer part of $x.$
The well-known Rademacher functions $r_o,r_1,...$ defined by $r_j(t)= (-1)^{[2^jt]}$ for $0 \leq t \leq 1,$ play an important role in many areas of analysis.
Let
$$
Rad_n = span[r_o,...,r_n].
$$
Applying Theorem \ref{numminimal} and Theorem \ref{opminimal} we will find a $N$-minimal projection from $X=Rad_m$ onto $Rad_n$ for any $ m >n$ with respect to the $L_p$-norm for $(1 \leq p < \infty).$ To do this, we need to define so-called dyadic addition on the interval $[0,1].$ Let $Q$
denote the set of all dyadic rationals from the interval $[0,1),$ i.e.
$$
Q = \{ x \in \mathbb{R}: x = \frac{p}{2^n}, p \in \mathbb{N},0 \leq p < 2^n\}.
$$
Note that any $ x \in [0,1]$ can be writtten in the form
$$
x = \sum_{k=0}^{\infty}x_k 2^{-(k+1)},
$$
where each $x_k=0$ or $1.$ For each $x \in [0,1] \setminus Q$ there is only one expression of this form.
When $ x \in Q$ there are two expressions of this form, one which terminates in 0's and the other one which terminates in 1's. By the dyadic expansion of $x \in Q$ we shall mean the one which terminates in 0's.
Now we can define the dyadic addition of two numbers $x,y \in [0,1]$ by:
$$
x \oplus y = \sum_{k=0}^{\infty}|x_k-y_k|2^{-(k+1)}.
$$
Notice that $ G= ([0,1], \oplus)$ is a group. Indeed, $ x\oplus 0 = x$ and $ x\oplus x = 0. $
Let us define a metric $d$ on $G$ by
$$
d(x,y) = \max\{ \sum_{k=0}^{\infty}|x_k-y_k|2^{-(k+1)}, |x-y| \},
$$
where $ x = \sum_{k=0}^{\infty}x_k 2^{-(k+1)}$ and $y = \sum_{k=0}^{\infty}y_k 2^{-(k+1)}$
It is easy to see that $(G,d)$ is a compact, topological group. Moreover, the normalized Haar measure on $G$ is precisely the Lebesgue measure on $[0,1].$
Also it is easy to see that for any $ n \in \mathbb{N}$ and $ x \in [0,1]$
\begin{equation}
\label{rad2}
r_n(x \oplus y)= r_n(x)r_n(y)
\end{equation}
provided $ x\oplus y \notin Q.$ Moreover, by the properties of Haar measures, for any $ g \in [0,1]$ the operator $T_g:L_p[0,1] \rightarrow L_p[0,1]$ given by
$$
(T_gf)x) = f(x\oplus g)
$$
is a linear, surjective isometry.
\\
Now we will show that if $ f_n, f \in L_p[0,1],$ $ \|f_n - f\|_p \rightarrow 0$ and $ |g_n - g| \rightarrow 0, $ then
\begin{equation}
\label{rad1}
\| T_{g_n}(f_n)-T_g(f)\|_p \rightarrow 0.
\end{equation}
To do this note that
$$
\| T_{g_n}(f_n)-T_g(f)\|_p \leq \| T_{g_n}(f_n)-T_{g_n}(f)\|_p + \| T_{g_n}(f)-T_g(f)\|_p.
$$
Observe that by changing variable from $x$ to $x\oplus g_n$ we get
$$
\| T_{g_n}(f_n)-T_g(f)\|_p^p = \int_{[0,1]}|f_n(x\oplus g_n)-f(x\oplus g_n)|^p d\mu(x)
$$
$$
=  \int_{[0,1]}|f_n(x)-f(x)|^p d\mu(x) = \|f_n-f\|_p^p \rightarrow 0.
$$
Notice that, if $ f $ is a continuous function (and hence uniformly continuous on $[0,1]$), since $g_n \rightarrow g, $ for any $ \epsilon >0$ there exists $n_o \in \mathbb{N}$ such that for any $x \in [0,1]$  and $ n \geq n_o$ $ |f(x\oplus g_n)-f(x\oplus g)| \leq \epsilon.$
Consequently, $ \|T_{g_n}f -T_gf\|_p \rightarrow 0$ for any $f \in C[0,1].$
By the Banach-Steinhaus Theorem, since $ 1 \leq p < \infty,$  it follows that  $\|T_{g_n}f - T_gf \|_p \rightarrow 0,$ which shows (\ref{rad1}).
\\
Note that, since $ Rad_n$ is a finite-dimensional subspace, $ \mathcal{P}(Rad_m, Rad_n)\neq \emptyset.$ By (\ref{rad2}) $T_g(Rad_n) \subset Rad_n$ for any $n \in \mathbb{N}.$
Consequently, by Theorem \ref{Rudin}, applying the fact that $ g^{-1} = g$ for any $g \in G,$ for any $ P \in \mathcal{P}(X, Rad_n)$
a projection
$$
Q_pf = \int_{[0,1]}(T_g P T_g)f d\mu(g) \in \mathcal{P}(X, Rad_n)
$$
commutes with $G.$ Now we show that there is exactly one projection from $X$ onto $Rad_n$ which commutes with $G.$ To do this, we show that for any $P \in \mathcal{P}(X, Rad_n)$
$Q_p(r_k) = 0$ for $ m \geq k > n.$ Accordingly we fix $x \in [0,1]$ and $g \in G$ with $x \oplus g \notin Q.$ Note that
$$
(T_g P T_gr_k)x = r_k(g)(T_gPT_gr_k)x = r_k(g)(T_g(\sum_{j=0}^n a_jr_j))x
$$
$$
= r_k(g) \sum_{j=0}^na_jr_j(x)r_j(g).
$$
Observe that $ \int_{[0,1]}r_j(g)r_k(g)d\mu(g)=0 $ if $ k\neq j. $
Since for any $ x \in [0,1]$
$$
\mu (\{g \in G; x\oplus g \in Q \}) =0,
$$
$$
(Q_pr_k)x = \int_{[0,1]} r_k(g) (\sum_{j=0}^na_jr_j(x)r_j(g))d\mu(g) = 0,
$$
which demonstrates our claim.
Consequently, for any $ P \in \mathcal{P}(Rad_m,Rad_n),$ and $ f \in Rad_m,$
$$
R_nf= Q_pf = \sum_{j=0}^n (\int_{[0,1]}r_j(t)f(t)d\mu(t))r_j
$$
is an $N$-minimal projection as considered in Theorem \ref{numminimal} and Theorem \ref{opminimal}. For more information about the nth Rademacher projection
$R_n$ the reader is referred to \cite{LES}.
\end{exm}
\begin{exm}
\label{ex6}
Let for $n \in \mathbb{N},$ $ X_n = \mathcal{L}(\mathbb{R}^n).$ Set $$
Y_n = \{ L \in X^n: L = L^{T}\}.
$$
Let us equip $X_n$ with an operator norm determined by any symmetric norm $ \| \cdot \|$ on $\mathbb{R}^n.$ (We say that $ \| \cdot \|$ is symmetric if
$$
\|\sum_{j=1}^n a_j e_j \| = \|\sum_{j=1}^n\epsilon_j a_{\sigma(j)}e_j\|
$$
for any $ a_1,...,a_n \in \mathbb{R},$ $ \epsilon_j \in \{ -1,1\}$ and any $\sigma$, where $\sigma$ is a permutation of $ \{ 1,...,n\}).$
\\
Set for $ L \in X_n$
$$
P(L) = (L+L^T)/2.
$$
It is clear that $P \in \mathcal{P}(X_n,Y_n).$ Moreover in \cite{MI1} (see also \cite{MI2}) it was shown applying Theorem \ref{Rudin} and Corollary \ref{unicity}
that $P$ is a minimal projection in $\mathcal{P}(X_n,Y_n).$ Hence $P$ is an $N$-minimal projection in any case considered in Theorem \ref{numminimal} and Theorem \ref{opminimal}.
\end{exm}
 \begin{exm}
\label{ex7}
   Let $(s_n(S))$  be a sequence associated with $S\in B(X,Y)$ satisfying certain conditions, one of which is $$ s_n(RST) \leq \| R \| s_n(S) \| T \| $$ for
$T\in B(X_0, X), \, S \in B(X,Y) \,\, \mbox{and} \,\, T\in B(Y,Y_0)$. For a complete definition of $s$-numbers we refer to  \cite{ap}. Furthermore,  $s_n(S)$ is called an $n$th $s$-number of $S$ and various $s$-numbers generate different operator ideals. For example, for $ 0 < p < \infty$, we call $S\in B(X,Y)$ a $\mathcal{C}_{p}^{s}$-operator\,\,if\,\, $(s_n(S)) \in \ell_{p}$.  Setting $$\| S \| _{p}^{s}: = \displaystyle \{ \sum_{1}^{\infty} s_n(S)^p \}^{1/p}$$ we obtain the quasi-normed operator ideal $(\mathcal{C}_{p}^{s}, \| S \| _{p}^{s})$. Thus one can apply Theorem 2.4 to the quasi norm $\| S \| _{p}^{s}$.
 \end{exm}
Note that in the context of Banach spaces there are several $s$-numbers, since there are certain rules assigning to every operator a decreasing sequence of numbers which characterize its approximation or compactness properties. The main examples of $s$-numbers are approximation numbers, Gelfand numbers, Kolmogorov numbers and Hilbert numbers. Thus, one can construct many operator ideals with quasi-norms depending on the particular $s$-number used.

\noindent
\mbox{~~~~~~~}Asuman G\"{u}ven AKSOY\\
\mbox{~~~~~~~}Claremont McKenna College\\
\mbox{~~~~~~~}Department of Mathematics\\
\mbox{~~~~~~~}Claremont, CA  91711, USA \\
\mbox{~~~~~~~}E-mail: aaksoy@cmc.edu \\ \\
\noindent
\mbox{~~~~~~~}Grzegorz LEWICKI\\
\mbox{~~~~~~~}Jagiellonian University\\
\mbox{~~~~~~~}Department of Mathematics\\
\mbox{~~~~~~~}\L ojasiewicza 6, 30-348, Poland\\
\mbox{~~~~~~~}E-mail: Grzegorz.Lewicki@im.uj.edu.pl\\\\

\end{document}